\documentclass[a4paper,11pt]{amsart}          
\usepackage{amsfonts,amsmath,latexsym,amssymb} 
\usepackage{amsthm}                
\usepackage{mathrsfs,upref}         
\usepackage{mathptmx}		    
\usepackage{float}
\usepackage{tikz}

\newtheorem{theorem}{Theorem}           
\newtheorem{lemma}[theorem]{Lemma}               
\newtheorem{corollary}[theorem]{Corollary}

\theoremstyle{definition}

\newtheorem{definition}[theorem]{Definition}
\newtheorem{example}[theorem]{Example}

\renewcommand{\Re}{\operatorname{Re}}
\renewcommand{\Im}{\operatorname{Im}}

\author{Anna Muranova}
 \address{Anna Muranova: IRTG 2235, 
University Bielefeld, Postfach 10 01 31, 
33501 Bielefeld, Germany} 

 \email{anna.muranova@gmail.com}

\title{On the notion of effective impedance}
\thanks{This research was supported by IRTG 2235 Bielefeld-Seoul ``Searching for the regular in the irregular:
Analysis of singular and random systems".}

\begin{document}

\maketitle



\begin{abstract}
It is known that electrical networks with resistors are related to the Laplace operator and random walk on weighted graphs. In this paper we consider more general electrical networks with coils, capacitors, and resistors. We give two mathematical models of such networks: complex-weighted graphs and graphs with weight from the ordered field of rational functions. The notion of effective impedance in both approaches is defined.
\end{abstract}

{\footnotesize
\smallskip

}



\section{Introduction}

It was shown in \cite{DS} and \cite{LPW} that there is a tight relation between electrical networks with resistors and weighted graphs. Ohm's and Kirchhoff's laws imply that the voltage  in the network is  a solution of the Dirichlet problem for the discrete Laplace operator on the weighted graph. Due to the maximum principle, the solution of the Dirichlet problem in this case exists and is unique (see, for example, \cite{G}). Hence, this provides  a mathematical justification of the notion of effective resistance as the inverse energy of the solution of the Dirichlet problem.

Consider now an electrical network of alternating current, that consists of impe\-dan\-ces (i. e. resistors, capacitors, and coils). In this case one rewrites Ohm's and Kirchhoff's laws  in the complex form (see \cite{F}) and obtains the Dirichlet problem with complex-valued coefficients. Maximum principle does not exist in this case, and solution of the Dirichlet problem may not exist or may be not unique, which creates difficulties in definition of the effective impedance. In this paper we propose two approaches of overcoming this difficulty.

In the first approach we show that, in the case of multiple solutions, all they have the same energy and, therefore, the effective impedance is well-defined. In the case of absence of solution the effective impedance is set to be $0$.

In the second approach, we consider the impedance of each edge as a rational function of the parameter $\lambda=i\omega$, where $\omega$ is the frequency of the current (see \cite{Br}), and 
use the fact, that rational functions of $\lambda$ form an ordered field (see \cite{W}). Fortunately, the maximum principle holds for the Laplace operator with weight from that field, which allows to solve uniquely the Dirichlet problem and, hence, to define the effective impedance as a rational function on $\lambda$.

The two notions of effective impedances coincide if the Dirichlet problem of the first approach has a unique solution. Otherwise, the question of identity of the two effective impedances remains open.

\section{Graphs with complex-valued weight}

Let $(V,E)$ be a connected graph, where $V$ is a set of vertices and $E$ is a set of (unoriented) edges. Unless otherwise is said, the set $V$ is always assumed finite.

Assume that each edge $xy$ is equipped with a resistance $R_{xy}$, inductance $L_{xy}$, and capacitance $C_{xy}$, where $R_{xy}, L_{xy}\in [0,+\infty)$ and $C_{xy}\in (0,+\infty]$, which correspond to physical resistor, inductor (coil), and capacitor.  Let $a_0, a_1\in V$ be two vertices attached to the source of alternating current of frequency $\omega>0$. 

The \emph{impedance} of the edge $xy$ is 
\begin{equation}\label{imp}
z_{xy}=R_{xy}+ L_{xy}i\omega+\frac{1}{C_{xy}i\omega }.
\end{equation}
Note that $\Re z_{xy}\ge 0$.
\begin{figure}[H]
\includegraphics[scale=1.0]{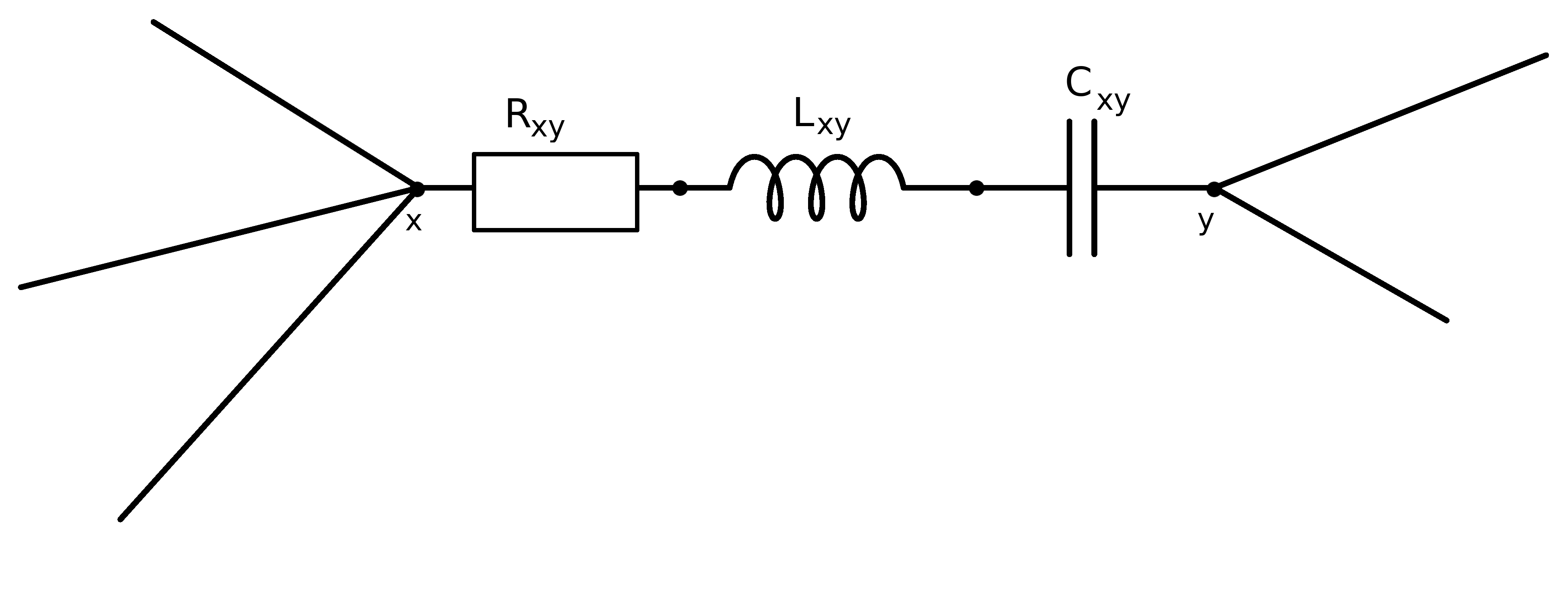}
\end{figure}

By complex Ohm's and Kirchhoff's laws the complex voltage $v:V\rightarrow \Bbb C$ satisfies
the following \emph{Dirichlet problem}:
\begin{equation}\label{GenDirpr}
 \begin{cases}
\sum_{y: y\sim x}\frac{v(y)-v(x)}{z_{xy}}=0 \mbox { on } V\setminus \{a_0,a_1\},
   \\
   v(a_0)=0,
   \\
 v(a_1)=1.
 \end{cases}
\end{equation}
Note that here and further in notations $\sum_y$ means $\sum_{y\in V}$. 

The physical voltage at time $t$ at the node $x$ is then $\Re(v(x)e^{i\omega t})$.

It will be convenient for us to use the inverse capacity:
\begin{equation*}
D_{xy}=\frac{1}{C_{xy}}\in  [0,+\infty),
\end{equation*}
as well as the \emph{admittance} $\rho_{xy}$:

\begin{align}
\rho_{xy}=&\frac{1}{z_{xy}}=\frac{i\omega}{ L_{xy}(i\omega)^2+R_{xy}(i\omega)+{D_{xy}}}=\frac{R_{xy}\omega^2+i(-L_{xy}\omega^3+D_{xy}\omega)}{(D_{xy}-L_{xy}\omega^2)^2+R^2_{xy}\omega^2}\label{rho}\\
=&\frac{\lambda}{ L_{xy}\lambda^2+R_{xy}\lambda+{D_{xy}}}\label{rhol},
\end{align}
where $\lambda=i\omega$ is pure imaginary, $\Im \lambda>0$ (i.e. $-i\lambda>0$). We always assume that for any edge 
\begin {equation*}
R^2_{xy}+L^2_{xy}+D^2_{xy}\ne 0.
\end {equation*}
Note that we can consider $\rho$ as a function from $V\times V$ to $\Bbb C$ by setting $\rho_{xy}=0$, if $xy$ is not an edge. We refer to the structure  $\Gamma=((V, E),\{\rho_{xy}\},a_0,a_1)$ as an \emph{(electrical) network}.

\begin{definition}
Define the \emph{Laplace operator} $\Delta_\rho$ as follows: for any function $f:V\rightarrow \Bbb C$
\begin{equation*}
\Delta_\rho f(x)=\sum_{y:y\sim x}(f(y)-f(x))\rho_{xy}=\sum_{y:y\sim x}(\nabla_{xy} f)\rho_{xy},
\end{equation*}
where
\begin{equation*}
\nabla_{xy} f=f(y)-f(x)
\end{equation*}
is a \emph{difference operator}.
\end{definition}
 Therefore, we can rewrite the Dirichlet problem \eqref{GenDirpr} as follows:
\begin{equation}\label{dirpr}
 \begin{cases}
\Delta_\rho f(x)=0 \mbox { on } V\setminus \{a_0,a_1\},
   \\
   v(a_0)=0, 
 \\
v(a_1)=1.
  
 \end{cases}
\end{equation}
Note that if $|V|=n$, then \eqref{dirpr} is a $n\times n$ system of linear equations. The existence and uniqueness of the solution of \eqref{dirpr} over $\Bbb C$ is not always the case.
\begin{example}\label{solutions}
Consider the network $\Gamma$ as at the figure below, where admittances are shown on each edge ($R,L,C>0$).

\begin{figure}[H]
\centering
\begin{tikzpicture}[auto,node distance=3cm,
                    thick,main node/.style={circle,draw,font=\sffamily\Large\bfseries}]

  \node[main node] (1) {$a_1$};
  \node[main node] (5) [above of=1]  {$z$};
  \node[main node] (4) [right of=5] {$y$};
  \node[main node] (3) [right of=4] {$x$};
  \node[main node] (2) [below of=3] {$a_0$};

  \path[every node/.style={font=\sffamily\small}]
    (2) edge node [bend right] {$\frac{1}{L \lambda +R}$} (1)
    (1) edge node [bend right] {${C \lambda }$} (5)
    (4) edge node [bend right] {$\frac{1}{L \lambda}$} (1)
    (5) edge node [bend right] {${C \lambda }$} (4)
    (3) edge node [bend right] {${C \lambda }$} (2)
    (2) edge node [bend right] {${C \lambda }$} (4)
    (4) edge node [bend right] {$\frac{1}{ L \lambda}$} (3);

\end{tikzpicture}
\end{figure}

The Dirichlet problem for this network is 
\begin{equation}\label{dirprex}
 \begin{cases}
\frac{(v(y)-v(x))}{L \lambda}+(v(a_0)-v(x)){C \lambda}=0,
\\
(v(z)-v(y)){C \lambda}+\frac{(v(a_1)-v(y))}{L \lambda}+(v(a_0)-v(y)){C \lambda}+\frac{(v(x)-v(y))}{L \lambda}=0,
   \\
(v(a_1)-v(z)){C \lambda}+(v(y)-v(z)){C \lambda}=0,
   \\
v(a_0)=0,\\
v(a_1)=1.
 \end{cases}
\end{equation}
The determinant of this linear system is 
\begin{equation*}
D=D(\lambda)=\frac{-C(3LC\lambda^2 +1)(LC\lambda^2+2)}{ L^2\lambda}
\end{equation*}
and it has two pure imaginary zeros with positive imaginary part: $\lambda_1=i\sqrt{\frac{2}{LC}}$, $\lambda_2=i\sqrt{\frac{1}{3LC}}$. In case $D(\lambda)\ne 0$ the solution of the Dirichlet problem \eqref{dirprex} is
\begin{equation*}
v=(v(x),v(y),v(z),v(a_0),v(a_1))=\left(\frac{1}{3LC\lambda^2+1},\frac{LC\lambda^2+1}{3LC\lambda^2+1},\frac{2LC\lambda^2+1}{3LC\lambda^2+1},0,1\right).
\end{equation*}
In the case $\lambda=\lambda_1$ the Dirichlet problem \eqref{dirprex} has infinitely many solutions
\begin{equation*}
v=\left(-2\tau+1,2\tau-1,\tau,0,1\right),\tau\in\Bbb C.
\end{equation*}
In the case $\lambda=\lambda_2$ the Dirichlet problem \eqref{dirprex} has no solution.

\end{example}

\begin{example}\label{nontrivial}
Consider the network $\Gamma$ as at the figure below, where admittances are shown on each edge.

\begin{figure}[H]
\centering
\begin{tikzpicture}[auto,node distance=3cm,
                    thick,main node/.style={circle,draw,font=\sffamily\Large\bfseries}]

  \node[main node] (1) {$a_0$};
  \node[main node] (2) [above right  of=1]  {$x$};
  \node[main node] (3) [below right of=1] {$y$};
  \node[main node] (4) [above right of=3] {$a_1$};

  \path[every node/.style={font=\sffamily\small}]
    (1) edge node [bend right] {$\frac{1}{\lambda}$} (2)
    (3) edge node [bend right] {$\lambda$} (1)
    (2) edge node [bend right] {$\lambda$} (4)
    (4) edge node [bend right] {$\frac{1}{\lambda }$} (3)
    (2) edge node [bend right] {$\frac{1}{\lambda+1 }$} (3);

\end{tikzpicture}
\end{figure}

The Dirichlet problem for this network is 
\begin{equation}\label{dirprex2}
 \begin{cases}
\frac{(v(a_0)-v(x))}{ \lambda }+\frac{(v(y)-v(x))}{ \lambda+1}+(v(a_1)-v(x)){ \lambda}=0,
   \\
(v(a_0)-v(y)){  \lambda }+\frac{(v(x)-v(y))}{ \lambda +1}+\frac{(v(a_1)-v(y))}{ \lambda }=0,\\

v(a_0)=0,\\
v(a_1)=1.
 \end{cases}
\end{equation}
The determinant of this linear system is 
\begin{align*}
D=D(\lambda)=&\frac{\lambda^5+\lambda^4+4\lambda^3+2\lambda^2+3\lambda+1}{\lambda^3+\lambda^2}\\=&\frac{(\lambda^2+1)(\lambda^3+\lambda^2+3\lambda+1)}{\lambda^2(\lambda+1)}
\end{align*}
and it has one pure imaginary zero, whose imaginary part is positive, $\lambda_0=i$. In case $D(\lambda)\ne 0$ the solution of the Dirichlet problem \eqref{dirprex2} is
\begin{align*}
v(x)&=\frac{\lambda^3+\lambda^2+\lambda}{\lambda^3+\lambda^2+3\lambda+1},\\
v(y)&=\frac{2 \lambda+1}{\lambda^3+\lambda^2+3\lambda+1}.
\end{align*}
In the case $\lambda=i$ the Dirichlet problem \eqref{dirprex2} has infinitely many solutions

\begin{align*}
v(x)&=\tau-1+i,\\
v(y)&=\tau,\tau\in\Bbb C.
\end{align*}

\end{example}

\begin{definition}\label{ZeffC}
Let $v(x)$ be a solution of the Dirichlet problem \eqref{dirpr}. Define the
\emph{effective impedance} of the network $\Gamma$ by 
\begin{equation*}
Z_{eff}=\frac{1}{\sum_{x:x\sim a_0}v(x)\rho_{xa_0}}
\end{equation*}
and the \emph{effective admittance} by 
\begin{equation*}
\mathcal{P}_{eff}=\frac{1}{Z_{eff}}={\sum_{x:x\sim a_0}v(x)\rho_{xa_0}}.
\end{equation*}
If  \eqref{dirpr} has no solution, then set $Z_{eff}=0$ and  $\mathcal{P}_{eff}=\infty$.
\end{definition}
Note that $Z_{eff}$ and $\mathcal{P}_{eff}$ take values in $\Bbb C\cup \{\infty\}$.
We will prove below (see Theorem \ref{Zeffmult}) , that in the case when \eqref{dirpr} has multiple solutions, the values $Z_{eff}$ and $\mathcal{P}_{eff}$ are independent of the choice of the solution $v$.

\begin{theorem}\label{nonzeroDet}
For any given network $\Gamma$ the determinant $D(\lambda)$ of the Dirichlet problem \eqref{dirpr} has a finite number of zeros in $\Bbb C$.
\end{theorem}

Hence, for all $\lambda\in \Bbb C$, except for a finite number of values, the Dirichlet problem \eqref{dirpr} has a unique solution.

The proof of Theorem \ref{nonzeroDet} will be given in Section 3. Note that the theorem is true just for networks (see Example \ref{minusLambda} below).

\begin{example}\label{minusLambda}
Consider the Dirichlet problem for the weighted graph at the figure below, where weights are shown on each edge. Note that here the weights of the edges $xa_1$ and $a_0y$ are not in the form \eqref{rhol}, and, therefore, this is {\it not a network}.
\begin{figure}[H]
\centering
\begin{tikzpicture}[auto,node distance=3cm,
                    thick,main node/.style={circle,draw,font=\sffamily\Large\bfseries}]

  \node[main node] (1) {$a_0$};
  \node[main node] (2) [above right  of=1]  {$x$};
  \node[main node] (3) [below right of=1] {$y$};
  \node[main node] (4) [above right of=3] {$a_1$};

  \path[every node/.style={font=\sffamily\small}]
    (1) edge node [bend right] {$\lambda$} (2)
    (3) edge node [bend right] {$-\lambda$} (1)
    (2) edge node [bend right] {$-\lambda$} (4)
    (4) edge node [bend right] {$\lambda$} (3)
    (2) edge node [bend right] {$1$} (3);

\end{tikzpicture}
\end{figure}

The Dirichlet problem for this graph is 
\begin{equation*}
 \begin{cases}
(v(a_0)-v(x))\lambda+(v(a_1)-v(x))(-\lambda)+(v(y)-v(x))=0,
   \\
(v(a_0)-v(y))(-\lambda)+(v(a_1)-v(y))\lambda+(v(x)-v(y))=0\\

v(a_0)=0,\\
v(a_1)=1.
 \end{cases}
\end{equation*}
The determinant of this system is 
\begin{equation*}
D=D(\lambda)\equiv 0
\end{equation*}
and the Dirichlet problem has infinitely many solutions
\begin{equation*}
v(x)=\tau, v(y)=\tau+\lambda,\tau\in\Bbb C
\end{equation*}
for any $\lambda$.
\end{example}

From physical point of view the effective impedance means that if we replace our entire network by a single edge connecting $a_0$ and $a_1$ with the impedance $Z_{eff}$, then the current in this single-edge network will be the same as in the original one.

\begin{lemma}[Green's formula]
Let $\Gamma$ be a network as above and let $\Omega$ be a non-empty subset of $V$. Then for any two functions $f, g:V\rightarrow \Bbb C$ the following identity is true:
\begin{equation}\label{Greenf}
\sum_{x\in \Omega}\Delta_\rho f(x){g(x)}=-\frac{1}{2}\sum_{x,y\in \Omega}(\nabla_{xy}f){(\nabla_{xy}g)}\rho_{xy}+\sum_{x\in \Omega}\sum_{y\in V\setminus\Omega}(\nabla_{xy}f){g(x)}\rho_{xy}.
\end{equation}

\end{lemma}

\begin{proof}
\begin{align*}
\sum_{x\in \Omega}\Delta_\rho f(x){g(x)}&=\sum_{x\in \Omega}\left(\sum_{y\in V}(f(y)-f(x))\rho_{xy}\right){g(x)}\\
&=\sum_{x\in \Omega}\sum_{y\in V}(f(y)-f(x)){g(x)}\rho_{xy}\\
&=\sum_{x\in \Omega}\sum_{y\in \Omega}(f(y)-f(x)){g(x)}\rho_{xy}+\sum_{x\in \Omega}\sum_{y\in V\setminus\Omega}(f(y)-f(x)){g(x)}\rho_{xy}\\
&=\sum_{y\in \Omega}\sum_{x\in \Omega}(f(x)-f(y)){g(y)}\rho_{xy}+\sum_{x\in \Omega}\sum_{y\in V\setminus\Omega}(\nabla_{xy}f){g(x)}\rho_{xy},\\
\end{align*}
where in the last line we have switched notation of the variables $x$ and $y$ in the first sum. Adding together the last two lines and dividing by $2$, we obtain \eqref{Greenf}.
\end{proof}

If $\Omega=V$, then $V\setminus\Omega$ is empty so that the last term in \eqref{Greenf} vanishes, and we obtain
\begin{equation}\label{VeqOm}
\sum_{x\in V}\Delta_\rho f(x){g(x)}=-\frac{1}{2}\sum_{x,y\in V}(\nabla_{xy}f)(\nabla_{xy}g)\rho_{xy}.
\end{equation}

\begin{corollary}
For any function $f:V\rightarrow \Bbb C$,
\begin{equation}\label{SumDeq0}
\sum_{x\in V}\Delta_\rho f(x)=0.
\end{equation}
\end{corollary}
\begin{proof}
Apply \eqref{VeqOm} for $g\equiv 1$.
\end{proof}

\begin{lemma}
For any solution $v$ of the Dirichlet problem \eqref{dirpr} we have
\begin{equation}\label{PeffD}
\sum_{x:x\sim a_0}v(x)\rho_{xa_0}=\Delta_\rho v(a_0)=-\Delta_\rho v(a_1)=\frac{1}{2}\sum_{x,y\in V}(\nabla_{xy}v)(\nabla_{xy}u)\rho_{xy},
\end{equation}
where $u:V\rightarrow \Bbb C$ is any function such that $u(a_0)=0$ and $u(a_1)=1$.
\end{lemma}
\begin{proof}
Using $v(a_0)=0$, we have
\begin{equation*}
\Delta_\rho v(a_0)=\sum_{x:x\sim a_0}(v(x)-v(a_0))\rho_{xa_0}=\sum_{x:x\sim a_0}v(x)\rho_{xa_0}
\end{equation*}
which proves the first identity in \eqref{PeffD}. Since by \eqref{SumDeq0} 
\begin{equation*}
\sum_{x\in V}\Delta_\rho f(x)=0
\end{equation*}
and $\Delta_\rho v(x)=0$ for all $x\in V\setminus\{a_0,a_1\}$, we obtain
\begin{equation*}
\Delta_\rho v(a_0)+\Delta_\rho v(a_1)=0
\end{equation*}
whence the second identity in \eqref{PeffD} follows.
Finally, to prove the third identity, we apply the Green's formula \eqref{VeqOm} and obtain 
\begin{equation*}
\frac{1}{2}\sum_{x,y\in V}(\nabla_{xy}v)(\nabla_{xy}u)\rho_{xy}=-\sum_{x}\Delta_\rho v(x)u(x)=-\Delta v(a_1),
\end{equation*}
because $\Delta_\rho v(x)=0$ for all $x\in V\setminus\{a_0,a_1\}$, while $u(a_0)=0$ and $u(a_1)=1$.
\end{proof}

\begin{theorem}\label{Zeffmult}
The value of the impedance $Z_{eff}$ does not depend on the choice of a solution $v$ of the Dirichlet problem \eqref{dirpr}. Besides, we have the identity
\begin{equation}\label{ccp}
\frac{1}{2}\sum_{x,y\in V}|\nabla_{xy} v|^2\rho_{xy}={\mathcal{P}_{eff}}
\end{equation}
(conservation of the complex power).
\end{theorem}
\begin{proof}
Let $v_1$ and $v_2$ be two solutions of \eqref{dirpr}. By \eqref{PeffD} we have
\begin{equation*}
\sum_{x:x\sim a_0}v_1(x)\rho_{xa_0}=\frac{1}{2}\sum_{x,y\in V}(\nabla_{xy}v_1)(\nabla_{xy}v_2)\rho_{xy}
\end{equation*}
and also
\begin{equation*}
\sum_{x:x\sim a_0}v_2(x)\rho_{xa_0}=\frac{1}{2}\sum_{x,y\in V}(\nabla_{xy}v_2)(\nabla_{xy}v_1)\rho_{xy},
\end{equation*}
whence the identity 
\begin{equation*}
\sum_{x:x\sim a_0}v_1(x)\rho_{xa_0}=\sum_{x:x\sim a_0}v_2(x)\rho_{xa_0}
\end{equation*}
follows. Hence, the admittance and impedance are independent of the choice of $v$.
Applying \eqref{PeffD} with $u=\overline {v}$, we obtain
\begin{equation*}
\frac{1}{2}\sum_{x,y\in V}|\nabla_{xy} v|^2\rho_{xy}=\frac{1}{2}\sum_{x,y\in V}(\nabla_{xy} v)(\overline{\nabla_{xy} v})\rho_{xy}={\mathcal{P}_{eff}}.
\end{equation*}
\end{proof}

By the physical meaning $\Re z_{xy}\ge 0$ and the effective impedance also is expected to have a non-negative real part. 
We prove this in a following theorem, using the conservation of complex power.

\begin{theorem}
For any finite network, we have
\begin{equation*}
\Re (Z_{eff})\ge 0.
\end{equation*}
Moreover, if $\Im (z_{xy})\le 0$ for any $xy\in E$ ($RC$-network), then $\Im (Z_{eff})\le 0$ and if $\Im (z_{xy})\ge 0$ for any $xy\in E$ ($RL$-network), then $\Im (Z_{eff})\ge 0$.
\end{theorem}

\begin{proof}
For any $z\in \Bbb C$, we have
\begin{equation*}
\Re z\ge 0\Leftrightarrow \Re \left(\frac{1}{z}\right)\ge 0,
\end{equation*}
because if $z=a+bi, a,b\in \Bbb R$, then 
\begin{equation}\label{invZ}
\frac{1}{z}=\frac{a}{a^2+b^2}-\frac{b}{a^2+b^2}i.
\end{equation}
Therefore, $\Re (Z_{eff})\ge 0$ is equivalent to $\Re\left(\mathcal{P}_{eff}\right)\ge 0$. 
From the left hand side of \eqref{ccp}  it is obvious that $\Re\left(\mathcal{P}_{eff}\right)\ge 0$ since $\Re (\rho_{xy})\ge 0$ for any $xy\in E$ by \eqref{rho}.\\

 We have $\Im (z_{xy})\le 0\Leftrightarrow \Im (\rho_{xy})\ge 0$ and $\Im (Z_{eff})\le 0\Leftrightarrow \Im \left(\mathcal{P}_{eff}\right)\ge 0$ by \eqref{invZ}. Due to \eqref{ccp}, $\Im (\rho_{xy})\ge 0$, for any $xy\in E$, implies $\Im \left(\mathcal{P}_{eff}\right)\ge 0$. The result for $RL$-network can be proved analogously.
\end{proof}

\begin{example}
The effective impedance for the network from Example \ref{solutions} is given by. 
\begin{equation*}
Z_{eff}=
\begin{cases}
R+i\sqrt{\frac{2L}{C}}, \lambda=i\sqrt{\frac{2}{LC}}, \mbox { case of multiple solutions };
\\
0, \lambda=i\sqrt{\frac{1}{3LC}}, \mbox { case of no solution };
\\
\frac{3L^2C\lambda^3+3RLC\lambda^2+L\lambda+R}{L^2C^2\lambda^4+RLC^2\lambda^3+5LC\lambda^2+2RC\lambda+1}, \mbox { for other  }\lambda, \mbox {such that} (-i\lambda)>0.
\end{cases}
\end{equation*}
It is easy to verify, that in this example the effective impedance is a continuous function on $\omega=-i\lambda\in (0,\infty)$.
\end{example}

\begin{example}
The effective impedance for the network from Example \ref{nontrivial} is given by
\begin{equation*}
Z_{eff}=
\begin{cases}
\frac{1}{2}-\frac{i}{2}, \lambda=i,
\\
\frac{\lambda^3+\lambda^2+3 \lambda+1}{3\lambda^2+2\lambda+1}, \mbox { in other cases } (-i\lambda>0).
\end{cases}
\end{equation*}
It is easy to verify, that here the effective impedance is again a continuous function on $\omega=-i\lambda\in (0,\infty)$.
\end{example}
\section{Network over an ordered field}
Let us consider the admittance $\rho_{xy}$ as a rational function of $\lambda$
\begin{equation}\label{rAdm}
\rho_{xy}=\frac{\lambda}{ L_{xy} \lambda^2  +R_{xy}\lambda +{D_{xy}}}
\end{equation}
with real coefficients.

Let us denote by $\Bbb R(\lambda)$ the set of all rational functions of $\lambda$ with real coefficients.
\begin{definition}\cite{W}
Define in $\Bbb R(\lambda)$ an \emph{order} `` $\succ$'' as follows: for any rational function
\begin{equation*}
f(\lambda)=\frac{a_n \lambda^n+\dots+ a_1 \lambda+a_0}{b_m \lambda^m+\dots+ b_1 \lambda+b_0}\in \Bbb R(\lambda)
\end{equation*}
with $a_n\ne 0, b_m\ne 0$, write
\begin{equation*}
f(\lambda)\succ 0 \mbox {, if } \frac{a_n}{b_m}> 0.
\end{equation*}
and 
\begin{equation*}
f(\lambda)\succ g(\lambda) \mbox {, if } f(\lambda)-g(\lambda)\succ 0.
\end{equation*}

\end{definition}
It is easy to check that $\succ$  is a total order and $(\Bbb R(\lambda),\succ)$ is an ordered field (see \cite{W}). Note that this field is non-Archimedean: $\lambda\succ n$ for any $n=\underbrace{1+\cdots+1}_{n}$.

Let $(K,\succ)$ be an arbitrary ordered field. We say that $k\in K$ is \emph{positive} if $k\succ 0$. For $k_1,k_2\in K$ we will write
\begin{equation*}
k_1\succeq k_2 \mbox {, if } k_1\succ k_2 \mbox { or } k_1=k_2.
\end{equation*}
Moreover, we will write 
\begin{equation*}
k_1\prec k_2 \mbox {, if } k_2\succ k_1
\end{equation*}
and
\begin{equation*}
k_1\preceq k_2 \mbox {, if } k_1\prec k_2 \mbox { or } k_1=k_2.
\end{equation*}

\begin{definition}
 \emph{A network over the ordered field $K$} is a structure 
\begin{equation*}
\Gamma=((V,E),\{\rho_{xy}\},a_0,a_1),
\end{equation*}
where $(V,E)$ is a connected graph, $\rho:E\rightarrow K$ is a positive function, and $a_0,a_1\in V$ are two fixed vertices.

\end{definition}

Note that we can consider $\rho$ as a function from $V\times V$ to $K$ by setting $\rho_{xy}=0$, if $xy$ is not an edge. Assume that the graph $(V,E)$ is locally finite. Then the weight $\rho_{xy}$ gives rise to a function on vertices as follows:
\begin{equation}
\rho(x)=\sum_y \rho_{xy},
\end{equation}
where the notation $\sum\limits_y$ means $\sum\limits_{y\in V}$.
Then $\rho(x)$ is called the \emph{weight of a vertex} $x$. By properties of the ordered field, we have $\rho(x)\succ 0$ for any $x\in V$.

\begin{definition}
For any function $f:V \rightarrow K$ the \emph{Laplace operator} $\Delta_\rho$  is defined as 
\begin{equation*}
\Delta_\rho f(x)=\sum_y (f(y)-f(x))\rho_{xy}=\sum_y (\nabla_{xy}f)\rho_{xy},
\end{equation*}
where 
\begin{equation*}
\nabla_{xy}f=f(y)-f(x)
\end{equation*}
is a \emph{difference operator}.
\end{definition}

From now on we assume that $(V, E)$ is a finite graph and $\Gamma$ is a network over the ordered field $K$ on this graph.

\begin{theorem}\label{ExUniq}
The following \emph{Dirichlet problem}:
\begin{equation}\label{dirprOF}
 \begin{cases}
\Delta_\rho v(x)=0 \mbox { on } V\setminus\{a_0,a_1\},
   \\
   v(a_0)=0,
   \\
v(a_1)=1.
 \end{cases}
\end{equation}
where $v:V\rightarrow K$ is an unknown function, has always a unique solution.
\end{theorem}

The key point for the proof of Theorem \ref{ExUniq} is the following lemma.
\begin{lemma}[A maximum/minimum principle]\label{minmax}
Let $B$ be a non-empty subset of $V$, such that $V\setminus B$ is also non-empty. Then, for any function $u:V\rightarrow K$, that satisfies $\Delta_\rho u(x)\succeq 0$ (i.e. $u$ is subharmonic) on $V\setminus B$, we have
\begin{equation}
\max_{V\setminus B} u\preceq \max_{B} u\
\end{equation}
and for any function $u:V\rightarrow K$, that satisfies $\Delta_\rho u(x)\preceq 0$ (i.e. $u$ is superharmonic) on $V\setminus B$, we have
\begin{equation}
\min_{V\setminus B} u \succeq \min_{B} u.
\end{equation}
\end{lemma}

\begin{proof}

It is enough to proof the first claim (then the second claim follows by changing $u$ to $-u$). 
Set 
\begin{equation*}
M  = \max_{V\setminus B} u,
\end{equation*}
and assume, that $M\succ \max_{B} u$. Let us consider the set
\begin{equation*}
S = \{x\in V:u(x)=M\}.
\end{equation*}
Clearly, $S\subset V\setminus B$ and $S$ is non-empty.\\

{\it Claim 1. If $x\in S$, then all neighbors of $x$ also belong to $S$.}

Indeed, we have  $\Delta_\rho u(x)\succeq 0$ which can be rewritten in the form
\begin{equation}\label{eqv(x)}
u(x)\preceq \sum_{y:y\sim x}\frac{\rho_{xy}}{\rho(x) }u(y).
\end{equation}
By properties of positive elements, we have
\begin{equation*}
\frac{\rho_{xy}}{\rho(x) }\succ 0 \mbox { for any } y\sim x.
\end{equation*}
Also, for any $y$ we have $u(y)\preceq M$ by the definition of maximum. Therefore,
\begin{equation}\label{1.1}
\frac{\rho_{xy}}{\rho(x) }u(y)=\frac{\rho_{xy}}{\rho(x) }M,\mbox{ if } u(y)=M,
\end{equation}
\begin{equation}\label{1.2}
\mbox{and }\frac{\rho_{xy}}{\rho(x) }u(y)\prec \frac{\rho_{xy}}{\rho(x) }M,\mbox{ if } u(y)\prec M,
\end{equation}
where the last line is true by properties of positive elements.
If there exist $y_0\sim x$ such that $u(y_0)\prec M$, then, summing up all the equalities \eqref{1.1} and inequalities \eqref{1.2}, we obtain
\begin{equation}\label{ineqv(x)}
\sum_{y\sim x}\frac{\rho_{xy}}{\rho(x) }u(y)\prec \sum_{y\sim x}\frac{\rho_{xy}}{\rho(x) } M.
\end{equation}
But 
\begin{equation*}
\sum_{y\sim x}\frac{\rho_{xy}}{\rho(x) } M=M=u(x), 
\end{equation*}
therefore, \eqref{ineqv(x)} is a contradiction with \eqref{eqv(x)}.\\

{\it Claim 2.} \emph{Let $S$ be a non-empty set of vertices of a connected graph (V,E) such that $x\in S$ implies that all neighbours of $x$ belong to $S$. Then $S=V$.}

Indeed, let $x\in S$ and $y$ be any other vertex. Then by the definition of connected graph, there is a path $\{x_k\}_{k=0}^n$ between $x$ and $y$, that is,
\begin{equation*}
x=x_0\sim x_1\sim x_2\sim\cdots\sim x_n=y.
\end{equation*}
Since $x_0\in S$ and $x_1\sim x_0$, we obtain $x_1\in S$. Since $x_2\sim x_1$, we obtain $x_2\in S$. By induction, we conclude that all $x_k\in S$, whence $y\in S$.\\

It follows from two claims that set $S$ must coincide with $V$, which is not possible since $u(x)\prec M$ for any $x\in B$. This contradiction shows that $M\preceq \max_{B} u$.
\end{proof}

\begin{proof}[Proof of the Theorem \ref{ExUniq}]

Let us first proof the uniqueness. If we have two solutions $v_1$ and $v_2$ of  \eqref{dirprOF},  then the difference $v=v_1-v_2$ satisfies the conditions
\begin{equation}
 \begin{cases}
\Delta_\rho v(x)=0 \mbox { on } V\setminus \{a_0,a_1\},
   \\
   v(x)=0 \mbox { on } \{a_0,a_1\},
   \\
 \end{cases}
\end{equation}
and, by Lemma \ref{minmax} 
\begin{equation*}
0=\max_{\{a_0,a_1\}} v \succeq \max_{V\setminus\{a_0,a_1\}} v \succeq  \min_{ V\setminus\{a_0,a_1\}} v\succeq \min_{\{a_0,a_1\}} v=0,
\end{equation*}
whence, $v\equiv 0$ since $v(a_0)=v(a_1)=0$.
Let us now prove the existence of a solution of \eqref{dirprOF}. For any $x\in V\setminus\{a_0,a_1\}$, rewrite the equation $\Delta_\rho v(x)=0$ in the form
\begin{equation}\label{Lv}
\sum_{\substack{y\sim x, \\  y\in V\setminus\{a_0,a_1\}}}\frac{\rho_{xy}}{\rho(x)}v(y)-v(x)=-\frac{\rho_{xa_1}}{\rho(x)}v(a_1)-\frac{\rho_{xa_0}}{\rho(x)}v(a_0).
\end{equation}
Let us denote by $\mathcal F$ the set of all functions $v$ on $V\setminus\{a_0,a_1\}$ with values in $K^{(n-2)}$, where $n=|V|$. Then the left hand side of \eqref{Lv} can be regarded as an operator in this space; let us denote it by $Lv$, that is
\begin{equation}
Lv(x)=\sum_{\substack{y\sim x,\\ y\in V\setminus\{a_0,a_1\}}}\frac{\rho_{xy}}{\rho(x)}v(y)-v(x),
\end{equation}
for all $x\in V\setminus\{a_0,a_1\}$. Rewrite the equation \eqref{Lv} in the form $Lv=h$, where $h$ is the right
hand side of \eqref{Lv}, which is a given function on $V\setminus\{a_0,a_1\}$. Note that $\mathcal F$ is a linear space over the field $K$.
Since the family $\{ {\bf 1}_{\{x\}}\}_{x\in V\setminus\{a_0,a_1\}}$ of indicator functions form a basis in $\mathcal F$, we obtain
that $\dim \mathcal F=n-2<\infty$. Hence, the operator $L:\mathcal F\rightarrow \mathcal F$ is a linear operator in
a finitely dimensional space, and the first part of the proof shows that $Lv = 0$ implies
$v= 0$ (indeed, just set $v(a_1) = 0$ and $v(a_0)=0$ in  \eqref{Lv}), that is, the operator $L$ is injective.
By Linear Algebra, any injective operator acting in the spaces of equal dimensions, must
be bijective. Hence, for any $h\in \mathcal F$ (in particular, for $h(x)=-\frac{\rho_{xa_1}}{\rho(x)}$), there is a solution, which finishes the proof.
\end{proof}

\begin{corollary}[Theorem \ref{nonzeroDet}]
For any given network $\Gamma$ the determinant $D(\lambda)$ of the Dirichlet problem \eqref{dirpr} has a finite number of zeros in $\Bbb C$.
\end{corollary}
\begin{proof}
For any given network the determinant of the Dirichlet problem \eqref{dirpr} is a rational function on $\lambda$ and, by Theorem \ref{ExUniq}, it is not constantly zero.
\end{proof}

\begin{corollary}
For the solution $v:V\rightarrow K$ of \eqref{dirprOF} the following inequality
\begin{equation}\label{v01}
0_K\preceq v(x) \preceq 1_K
\end{equation}
is true for any $x\in V$.
\end{corollary}
\begin{proof}
Apply Lemma \ref{minmax}  for  $B=\{a_0,a_1\}$.
\end{proof}
Now we can define the effective impedance of the network over the ordered field $K$.
\begin{definition}\label{ZeffOF}
Let $v$ be a solution of the Dirichlet problem \eqref{dirprOF} for the network. Then define the \emph{effective impedance} by
\begin{equation}\label{Yeff0} 
Z_{eff}=\frac{1}{\sum_x v(x)\rho_{xa_0}},
\end{equation}
The quantity 
\begin{equation*}
\mathcal{P}_{eff}=\frac{1}{Z_{eff}}={\sum_x v(x)\rho_{xa_0}}
\end{equation*}
is called the \emph {effective admittance}.
\end{definition}

Since by Theorem \ref{ExUniq} the Dirichlet problem \eqref{dirprOF} has exactly one solution over the field $K$, the effective impedance is always well-defined.

Moreover, by \eqref{v01} we have  $\mathcal{P}_{eff}\succeq 0$ and, hence, $Z_{eff}\succeq 0$.

\begin{theorem}[Green's formula]
Let $\Gamma$ be a network over the ordered field $K$ with the vertex set $V$, and let $\Omega$ be a non-empty subset of $V$. Then, for any two functions $f,g$ on $V$,
\begin{equation}\label{GreenfOF}
\sum_{x\in \Omega}\Delta_\rho f(x)g(x)=-\frac{1}{2}\sum_{x,y\in \Omega}(\nabla_{xy}f)(\nabla_{xy}g)\rho_{xy}+\sum_{x\in \Omega}\sum_{x\in V\setminus\Omega}(\nabla_{xy}f)g(x)\rho_{xy}.
\end{equation}
If $\Omega=V$, then the last term in \eqref{GreenfOF} vanishes, and we obtain
\begin{equation}\label{VeqOmOF}
\sum_{x\in V}\Delta_\rho f(x)g(x)=-\frac{1}{2}\sum_{x,y\in V}(\nabla_{xy}f)(\nabla_{xy}g)\rho_{xy}
\end{equation} 
\end{theorem}
\begin{proof}
\begin{align*}
\sum_{x\in \Omega}\Delta_\rho f(x)&g(x)=\sum_{x\in \Omega}\left(\sum_{y\in V}(f(y)-f(x))\rho_{xy}\right)g(x)\\
&=\sum_{x\in \Omega}\sum_{y\in V}(f(y)-f(x))g(x)\rho_{xy}\\
&=\sum_{x\in \Omega}\sum_{y\in \Omega}(f(y)-f(x))g(x)\rho_{xy}+\sum_{x\in \Omega}\sum_{y\in V\setminus\Omega }(f(y)-f(x))g(x)\rho_{xy}\\
&=\sum_{y\in \Omega}\sum_{x\in \Omega}(f(x)-f(y))g(y)\rho_{xy}+\sum_{x\in \Omega}\sum_{y\in V\setminus\Omega}(\nabla_{xy}f)g(x)\rho_{xy},\\
\end{align*}
where in the last line we have switched notation of the variables $x$ and $y$ in the first sum. Adding together the last two lines and dividing by $2$ (it is possible, since any ordered field has characteristic 0, see \cite{W}), we obtain
$$
\sum_{x\in \Omega}\Delta_\rho f(x)g(x)=-\frac{1}{2}\sum_{x\in \Omega}\sum_{y\in \Omega}(\nabla_{xy}f)(\nabla_{xy}g)\rho_{xy}+\sum_{x\in \Omega}\sum_{y\in V\setminus\Omega }(\nabla_{xy}f)g(x)\rho_{xy},
$$
which was to be proved.
\end{proof}

\begin{corollary}
For any function $f:V\rightarrow K$,
\begin{equation}\label{SumDeq0OF}
\sum_{x\in V}\Delta_\rho f(x)=0.
\end{equation}
\end{corollary}
\begin{proof}
Apply \eqref{VeqOmOF} for $g\equiv 1$.
\end{proof}

\begin{lemma}\label{PeffDOF}
For any network we have
\begin{equation}\label{DeltaPor}
\mathcal{P}_{eff}=\Delta_\rho v(a_0)=-\Delta_\rho v(a_1),
\end{equation}
where $v$ is the solution of the Dirichlet problem \eqref{dirprOF}.
\end{lemma}
\begin{proof}
Using $v(a_0)=0$, we obtain
\begin{equation}
\Delta_\rho v(a_0)=\sum_{y:y\sim a_0}(v(y)-v(a_0))\rho_{a_0y}=\sum_{y:y\sim a_0}v(y)\rho_{a_0y}=\mathcal{P}_{eff}.
\end{equation}
The second equality in \eqref{DeltaPor} follows from \eqref{SumDeq0OF}, since $v$ is the solution of the Dirichlet problem \eqref{dirprOF}.
\end{proof}

\begin{theorem}[Conservation of power over the ordered field]
Let $v$ be the solution of the Dirichlet problem \eqref{dirprOF} for network $\Gamma$ over the ordered field $K$. Then
\begin{equation}\label{cpOF}
\frac{1}{2}\sum_{x,y\in V}(\nabla_{xy} v)^2\rho_{xy}=\mathcal{P}_{eff}.
\end{equation}
\end{theorem}
\begin{proof}
Applying \eqref {VeqOmOF} to the left hand side of \eqref{cpOF} we obtain
\begin{align*}
\frac{1}{2}\sum_{x,y\in V}(\nabla_{xy} v)^2\rho_{xy}&=-\sum_{x\in V}\Delta_\rho v(x)v(x)\\
&=-\sum_{x\in V\setminus\{a_0,a_1\}}\Delta_\rho v(x)v(x)-\Delta_\rho v(a_0)v(a_0)+\Delta_\rho v(a_1)v(a_1)\\
&=-\Delta_\rho v(a_1),
\end{align*}
since $v$ is the solution of \eqref{dirprOF}.
The statement \eqref {cpOF} is proved due to Lemma \ref{PeffDOF}.
\end{proof}

\begin{theorem}[Dirichlet/Thomson's principle]
Let $v$ be the solution of the Dirichlet problem \eqref{dirprOF} for the network $\Gamma$ over the ordered field $K$. Then
for any other function $f:V\rightarrow K$ such that $f(a_0)=0$ and $f(a_1)=1$, the following inequality  holds:
\begin{equation}\label{energy}
\frac{1}{2}\sum_{x,y\in V}(\nabla_{xy} v)^2\rho_{xy}\preceq\frac{1}{2}\sum_{x,y\in V}(\nabla_{xy} f)^2\rho_{xy}
\end{equation}
\end{theorem}

\begin{proof}
Let $g=f-v$. Then $g(a_0)=g(a_1)=0$. Therefore,

\begin{align*}
\begin{split}
\frac{1}{2}\sum_{x,y\in V}(\nabla_{xy} f)^2\rho_{xy}=&\frac{1}{2}\sum_{x,y\in V}(\nabla_{xy} (g+v))^2\rho_{xy}=\frac{1}{2}\sum_{x,y\in V}(\nabla_{xy}g+\nabla_{xy}v)^2\rho_{xy}\\
=&\frac{1}{2}\sum_{x,y\in V}((\nabla_{xy}g)^2+2(\nabla_{xy}g)(\nabla_{xy}v)+(\nabla_{xy}v)^2)\rho_{xy}\\
=&\frac{1}{2}\sum_{x,y\in V}(\nabla_{xy} v)^2\rho_{xy}+\frac{1}{2}\sum_{x,y\in V}(\nabla_{xy} g)^2\rho_{xy}+\sum_{x,y\in V}(\nabla_{xy} v)(\nabla_{xy} g)\rho_{xy},
\end{split}
\end{align*}
where the last term vanishes by Green's formula \eqref{GreenfOF}, since $g(a_0)=g(a_1)=0$ and $v$ is the solution of the Dirichlet problem\eqref{dirprOF} and the second term is greater then zero whenever $g\not\equiv 0$. Therefore, \eqref{energy} is proved and an equality is attained if and only if $f\equiv v$.
\end{proof}

\section{Comparison of two definitions of $Z_{eff}$}
Denote by $Z_{eff}^{(1)}(\lambda)$ the effective impedance defined in Section 2, that we from now on will consider as a function of $\lambda=i\omega$.

The effective impedance from Section 3 for the field $K=\Bbb R(\lambda)$ will be denoted by $Z_{eff}^{(2)}(\lambda)$. Note that it was already defined as a rational function of $\lambda$.

Of course, the arises question is whether 
\begin{equation}\label{equaZ}
Z_{eff}^{(1)}(\lambda)=Z_{eff}^{(2)}(\lambda)
\end{equation}
for all $\lambda$, such that $-i\lambda>0$.

The unique solution $v(\lambda)$ of the Dirichlet problem \eqref{dirprOF} can be found by Cramer's rule applied in the field $\Bbb R(\lambda)$. Note that $v(\lambda)(x)$ is a rational function on $\lambda$ for any $x\in V$ and, hence, it is continuous on $\lambda$. Therefore, $Z_{eff}^{(2)}(\lambda)$ is a continuous function on $\lambda$ (with values in $\Bbb C\cup\{\infty\}$).

By Cramer's rule, applied in $\Bbb C$, the function $Z_{eff}^{(1)}(\lambda)$ is also a rational function of $\lambda$ at all $\lambda$, where the determinant $D(\lambda)$ of the Dirichlet problem \eqref{dirpr} does not vanish. Moreover, at those $\lambda$, where $D(\lambda)\ne 0$ the equality \eqref{equaZ} is true by Cramer's rule.

By Theorem \ref{ExUniq}, $D(\lambda)$ vanishes only at finitely many values of $\lambda$, therefore, the identity \eqref{equaZ} will be true for all $-i\lambda>0$ if we know that $Z_{eff}^{(1)}(\lambda)$ is continuous in $\lambda$.

However, it is not obvious for those $\lambda$, where \eqref{dirpr} has multiple solutions or no solution.

The question should definitely be restricted just to the Dirichlet problem, which arises from electrical networks, and to the case of pure imaginary $\lambda$, as the following two examples show.

\begin{example} \label{nonPosW}
Consider the Dirichlet problem for the network at the figure below, where weights are shown on each edge. Note, that here the weight of the edge $ya_1$ is positive function, but it is \emph{not} in the form \eqref{rhol}.
\begin{figure}[H]
\centering
\begin{tikzpicture}[auto,node distance=2.5cm,
                    thick,main node/.style={circle,draw,font=\sffamily\Large\bfseries}]

  \node[main node] (1) {$a_0$};
  \node[main node] (2) [above right of=1] {$x$};
  \node[main node] (3) [below right of=1] {$y$};
  \node[main node] (4) [above right of=3] {$a_1$};

  \path[every node/.style={font=\sffamily\small}]
    (1) edge node [bend right] {${\lambda}$} (2)
    (2) edge node [bend right] {$\frac{1}{\lambda}$} (4)
    (3) edge node [bend right] {$1$} (1)
    (4) edge node [bend right] {$\lambda+\frac{1}{\lambda}-1=\frac{\lambda^2-\lambda+1}{\lambda}$} (3);

\end{tikzpicture}
\end{figure}

The Dirichlet problem for this network is 
\begin{equation*}
 \begin{cases}

(v(a_0)-v(x))\lambda+\frac{v(a_1)-v(x)}{\lambda}=0,

   \\
(v(a_0)-v(y))+(v(a_1)-v(y))(\lambda+\frac{1}{\lambda}-1)=0,\\

v(a_0)=0,\\
v(a_1)=1.
 \end{cases}
\end{equation*}
The determinant of this system is 
\begin{equation*}
D=D(\lambda)=\frac{1}{\lambda^2}(\lambda^2+1)^2
\end{equation*}
and its zeros are $i$ and $-i$. 

In case $D(\lambda)\ne 0$ the solution of the Dirichlet problem is
\begin{equation*}
v(x)=\frac{1}{\lambda^2+1}, v(y)=\frac{\lambda^2-\lambda+1}{\lambda^2+1}.
\end{equation*}
and it has no finite limit as $\lambda\rightarrow i$.

But the effective impedance in this case is $Z_{eff}^{(1)}(\lambda)=1$. Therefore, $Z_{eff}^{(2)}(\lambda)\equiv 1$

The Dirichlet problem in the case $\lambda= i$ is 
\begin{equation*}
 \begin{cases}

(v(a_0)-v(x))i-(v(a_1)-v(x))i=0,

   \\
(v(a_0)-v(y))-(v(a_1)-v(y))=0,\\

v(a_0)=0,\\
v(a_1)=1.
 \end{cases}
\end{equation*}
and, obviously, has no solutions.
Therefore, $Z_{eff}^{(1)}(i)=0$ by definition. Hence  $Z_{eff}^{(1)}(\lambda)$ is not continuos at the point $\lambda=i$ and $Z_{eff}^{(1)}(i)\ne Z_{eff}^{(2)}(i)$.

\end{example}

\begin{example} \label{ComplexOmega}
Consider the Dirichlet problem for the network at the figure below, where admittances are shown on each edge.
\begin{figure}[H]
\centering
\begin{tikzpicture}[auto,node distance=2.5cm,
                    thick,main node/.style={circle,draw,font=\sffamily\Large\bfseries}]

  \node[main node] (1) {$a_0$};
  \node[main node] (2) [right of=1]  {$y$};
  \node[main node] (3) [below of=2] {$z$};
  \node[main node] (4) [above of=2] {$x$};
  \node[main node] (5) [right of=2] {$a_1$};

  \path[every node/.style={font=\sffamily\small}]
    (1) edge node [bend right] {$\frac{1}{\lambda}$} (4)
    (1) edge node [bend right] {$\lambda$} (2)
    (4) edge node [bend right] {$1$} (2)
    (4) edge node [bend right] {$\lambda$} (5)
    (2) edge node [bend right] {$\frac{1}{\lambda}$} (5)
    (3) edge node [bend right] {$\lambda$} (1)
    (5) edge node [bend right] {$\frac{1}{\lambda}$} (3);

\end{tikzpicture}
\end{figure}

The Dirichlet problem for this network is 
\begin{equation*}
 \begin{cases}
\frac{(v(a_0)-v(x))}{\lambda}+(v(a_1)-v(x))\lambda+(v(y)-v(x))=0,
   \\
(v(a_0)-v(y))\lambda+\frac{(v(a_1)-v(y))}{\lambda}+(v(x)-v(y))=0,
\\
(v(a_0)-v(z))\lambda+\frac{(v(a_1)-v(z))}{\lambda}=0,\\

v(a_0)=0,\\
v(a_1)=1.
 \end{cases}
\end{equation*}
The determinant of this system is 
\begin{align*}
D=D(\lambda)=&-\left(\frac{1}{\lambda^2}+\frac{2}{\lambda}+2+2\lambda+\lambda^2\right)\left(\lambda+\frac{1}{\lambda}\right)\\
=&-\frac{1}{\lambda^3}(\lambda+1)^2(\lambda^2+1)^2,
\end{align*}
and it is easy to see, that $\lambda=\pm i$ and $\lambda=-1$ are its zeros. 

In case $D(\lambda)\ne 0$ the solution of the Dirichlet problem is
\begin{equation*}
v(\lambda)=(v(x),v(y),v(z),v(a_0),v(a_1))=\left(\frac{\lambda}{1+\lambda}, \frac{1}{1+\lambda},\frac{1}{1+\lambda^2},0,1\right).
\end{equation*}

The effective impedance in these cases is 
\begin{equation}\label{zeffex2}
Z_{eff}^{(1)}(\lambda)=Z_{eff}^{(2)}(\lambda)=\frac{\lambda^2+1}{\lambda^2+\lambda+1}.
\end{equation}

Note that the finite limit of $v$ does not exist when $\lambda$ goes to $i$ or $\lambda$ goes to $-1$.

The Dirichlet problem in the case $\lambda=i$ is
\begin{equation*}
 \begin{cases}
-(v(a_0)-v(x)){i}+(v(a_1)-v(x))i +(v(y)-v(x))=0,
   \\
(v(a_0)-v(y))i -(v(a_1)-v(y))i+(v(x)-v(y))=0,
\\
(v(a_0)-v(z))i -(v(a_1)-v(z))i=0,\\

v(a_0)=0,\\
v(a_1)=1.
 \end{cases}
\end{equation*}
and has no solution, which by definition of effective impedance implies $Z_{eff}^{(1)}(i)=0$. It is easy to see, that $Z_{eff}^{(2)}(i)=0$. Therefore  $Z_{eff}^{(1)}(i)=Z_{eff}^{(2)}(i)$ and $Z_{eff}^{(1)}(\lambda)$ is continuous at the point $\lambda=i$.

The Dirichlet problem in the case $\lambda=-1$ is
\begin{equation*}
 \begin{cases}
-(v(a_0)-v(x))-(v(a_1)-v(x))+(v(y)-v(x))=0,
   \\
-(v(a_0)-v(y))-{(v(a_1)-v(y))}+(v(x)-v(y))=0,
\\
-(v(a_0)-v(z))-{(v(a_1)-v(z))}=0,\\

v(a_0)=0,\\
v(a_1)=1.
 \end{cases}
\end{equation*}
and it has multiple solutions
\begin{equation*}
v=(v(x),v(y),v(z),v(a_0),v(a_1))=\left(\tau, 1-\tau,\frac{1}{2},0,1\right),\tau \in\Bbb C.
\end{equation*}
One can calculate, that, $Z_{eff}^{(1)}(-1)=-\frac{2}{3}$. But from \eqref{zeffex2} follows, that $Z_{eff}^{(1)}(-1)=2$. Therefore, \eqref{equaZ} fails at the point $\lambda=-1$ and $Z_{eff}^{(1)}(\lambda)$ is not continuous at this point.

\end{example} 

\section*{Acknowledgement}
The author thanks her scientific advisor, Professor Alexander Grigor'yan, for helpful discussions on the topic.

\bigskip

\end{document}